\documentclass[12pt]{article}

\usepackage[margin=1in]{geometry}  % set the margins to 1in on all sides
\usepackage{graphicx}              % to include figures
\usepackage{amsmath}               % great math stuff
\usepackage{amsfonts}              % for blackboard bold, etc
\usepackage{amsthm}                % better theorem environments

\newtheorem{definition}{Definition}
\newtheorem{lem}{Lemma}
\newtheorem{prop}{Proposition}
\newtheorem{cor}{Corollary}
\newtheorem{conj}{Question}

\begin{document}

%\nocite{*}

\title{Group With Maximum Undirected Edges in Directed Power Graph Among All  Finite Non-Cyclic\\ Nilpotent Groups}

\author{ P. Darbari${}^2$ and B. Khosravi${}^{1,2}$\thanks{The second author was supported in part
by a grant from IPM (No. 92050120}\\
 $^{1}$School of Mathematics,
\\ Institute for Research in Fundamental sciences  (IPM),\\ P.O.Box:
19395--5746, Tehran, Iran \\
 $^{2}$Dept. of Pure  Math.,  Faculty  of Math. and Computer Sci. \\
Amirkabir University of Technology (Tehran Polytechnic)\\ 424, Hafez
Ave., Tehran 15914, IRAN \\
email: khosravibbb@yahoo.com}

\maketitle


\begin{abstract}

In [Curtin and Pourgholi, A group sum inequality and its application to power graphs, J. Algebraic Combinatorics, 2014], it is proved that among all directed power graphs of groups of a given order $ n $, the directed power graph of cyclic group of order $ n $ has the maximum number of undirected edges. In this paper, we continue their work and we determine a non-cyclic nilpotent group of an odd order $ n $ whose directed power graph has the maximum number of undirected edges among all non-cyclic nilpotent groups of order $ n $.

 We next determine non-cyclic $p$-groups whose undirected power graphs have the maximum number of edges among all groups of the same order.

\end{abstract}

\section{Introduction}
Many authors studied the directed (undirected) power graphs of
finite groups (see \cite{Cam, Cameron, Ivy}). This is an interesting
question in this field to determine groups with maximum edges or
maximum directed edges in their directed power graphs or undirected
power graphs.

 Isaacs
et al. in \cite{isa} proved that among all groups of order $n$, the
summation of the element orders of the cyclic group of order $n$ is
maximum. This is equivalent to this fact that among all groups of
order $n$ the number of directed edges in the directed power graph
of cyclic group of order $n$ is maximum. Later in \cite{SKeller} the
second maximum sum of element orders of finite nilpotent groups was
determined.

 Curtin et al. in \cite{Kevi} showed that
among directed power graphs of finite groups of a given order $ n $,
the cyclic group of order $ n $ has the maximum number of undirected
edges. In another studies \cite{Kevin}, they showed that the same is
true for undirected power graphs.  It is a natural question that
which non-cyclic groups have the maximum number of undirected edges
in their directed power graphs. In this paper we determine a
non-cyclic group $ G $ of an odd order such that among all
non-cyclic nilpotent groups of order $ \vert G\vert $ has maximum
number of undirected edges in its directed power graph. In the rest
of the paper, we determine finite non-cyclic $ p $-groups whose
undirected power graphs have the maximum number of edges.
\begin{definition}
Let $ G $ be a finite group. Let $ \langle g\rangle $ denote the cyclic subgroup of $ G $ generated by $ g\in G $.
\begin{enumerate}
\item[(i)]
The directed power graph $ \overrightarrow{\mathcal{G}}(G) $ of $ G $ is the directed graph whose vertex set is elements of $ G $ and for two distinct vertices $ x,y\in G $ there is an arc from $ x $ to $ y $ if and only if $ y=x^m $, for some positive integer $ m $; hence the set of directed edges is $ \overrightarrow{E}(G)=\{(g, h) \,| \, g, h \in G, h \in \langle g \rangle - \{g\}\} $ and the set of undirected edges is $\overleftrightarrow{E}(G)= \{\{g,h\}\} | \, h,g\in G,\, h\in\langle g\rangle-\{g\}\, and \,  g\in \langle h\rangle-\{h\}\}$.
\item[(ii)]
The undirected power graph (or power graph) $ \mathcal{G}(G) $ of $ G $ is the undirected graph whose vertex
set is the elements of $ G $ and two vertices being adjacent if one is a power of the other; hence the set of edges is $E(G)=\{(g, h) \,| \, h,g\in G, \, g\in\langle h\rangle -\{h\} \, or\, h\in \langle g\rangle-\{g\} \} $.
\end{enumerate}
\end{definition}
Let $ G $ be a group. We denote the order of an element $ a $ in a group $ G $ by $ o(a) $. Let $ \phi $ denotes the Euler totient function. Throughout the paper we use $ C_m $ to denote the cyclic group of order $ m $. Also, if $ n\geq 3 $ and $ p $ is an odd prime number, then we suppose that $ M_{n,p}= \langle a,b\vert \, b^p=1=a^{p^{n-1}},\, a^b=a^{1+p^{n-2}}\rangle$.
\\
\par
Let $ g$ and $ h $ be distinct elements of $ G $. There is an undirected edge between $ g $ and $ h $ when they generate the same subgroup of $ G $; hence the number of undirected edges which involve the vertex $ g $ is precisely $ \phi(o(g))-1$. Thus
\begin{equation*}
\vert \overleftrightarrow{E}(G)\vert =\frac{1}{2}\Sigma_{g\in G}(\phi(o(g))- 1),
\end{equation*}

Now let $  \phi (G)= \Sigma_{g\in G}{\phi(o(g))}$. In order to find non-cyclic nilpotent groups of a given order with the maximum value of $ \vert \overleftrightarrow{E}(G)\vert$, we should find  non-cyclic nilpotent groups with maximum value of $ \phi$. It was shown in \cite{Kevi} that among all groups of a given order, the cyclic group has the maxmimum value of $ \phi $. In this paper, among all non-cyclic nilpotent groups of an odd order, we determine a group with the maximum value of $ \phi $. We note that if the order of a nilpotent group $ G $ is free square, then $ G $ is cyclic. Hence, our main result is the following theorem.
\bigskip

\textbf{\\Main Theorem.}
{\it Let $ n=p^{\alpha_1}_1p^{\alpha_2}_2\cdots p^{\alpha_k}_k $ be a positive odd integer which is not square free and  $  p_1 < p_2 < \cdots < p_k  $ are primes. Set $  s = min\{1,\ldots,k\}$ such that $ \alpha_s>1 $. Suppose that $ G $ is a non-cyclic nilpotent group of order $ n $. Then $ \phi(G)\leq\phi(C_{\frac{n}{p_s}}\times C_{p_s}) $. Thereore
\begin{equation*}
 \vert \overleftrightarrow{E}(C_{\frac{n}{p_s}}\times C_{p_s})\vert\geq \vert\overleftrightarrow{E}(G) \vert ,
\end{equation*}
i.e., the directed power graph of $ C_{\frac{n}{p_s}}\times C_{p_s} $ has the maximum number of undirected edges among all non-cyclic nilpotent groups of order $ n $.
}

\section{Proof of Main Theorem}

For the proof of the main result, we use the following lemma.

\begin{lem}\cite[Lemma 3.1]{Kevi}
Let $ U $ and $ T $ be finite groups with $ (\vert U\vert,\vert T\vert)=1 $, and let $ G=U\times T$ be the direct product of $ U $ and $ T$. Then $ \phi(G) =\phi(U)\phi(T)$.
\end{lem}

\begin{prop}\label{Bj}
Among all finite non-cyclic groups of order $ p^n $ where $ p $ is an odd prime number, the groups $ C_{p^{n-1}} \times C_p $ and $ M_{n,p}$ have the maximum value of $ \phi $.
\end{prop}
\begin{proof}

Let $ G $ be a non-cyclic $p$-group of order $ p^n $. For every non-identity element $ g \in G$, we have $ \phi(o(g))= o(g)-\dfrac{o(g)}{p} $. Therefore
\begin{eqnarray*}
\phi(G)&=&\Sigma_{g\in G}\phi(o(g))=\Sigma_{g\in G - \{ e\}}(o(g)-\dfrac{o(g)}{p}) +1
\\
&=&\Sigma_{g\in G}o(g)-\Sigma_{g\in G}\frac{o(g)}{p}+\dfrac{1}{p}
\\
&=&(1-\dfrac{1}{p})\Sigma_{g\in G}o(g)+\dfrac{1}{p} .
\end{eqnarray*}

Hence, if $ \Sigma_{g\in G}o(g)$ has the maximum value, $ \phi(G) $ has the maximum value, too. It follows from Proposition 2.3 in \cite{SKeller} that $ \Sigma_{g\in G}o(g) $ has the maximum value when $ G\cong C_{p^{n-1}}\times C_p $ or $ G\cong M_{n,p} $. The proof is complete.
\end{proof}

\begin{cor}
Let $ G $ be a non-cyclic group of order $ p^n $ where $ p $ is an odd prime. Thus
\begin{equation*}
\vert\overleftrightarrow{E}(C_{p^{n-1}} \times C_p)\vert=\vert\overleftrightarrow{E}(M_{n,p})\vert\geq \vert\overleftrightarrow{E}(G)\vert .
\end{equation*}
Equality holds when $ G \cong C_{p^{n-1}}\times C_p  $ or $ G\cong M_{n,p} $.
\end{cor}

\begin{lem}\label{Mj}
Let $ p $ be a prime number and  $ m \geq 2$. Then

\begin{enumerate}
\item[(i)]
$\phi(C_{p^m})=\phi(C_{p^{m-1}})+\phi(p^m)^2  $;

\item[(ii)]
$ \phi(C_{p^{m-1}}\times C_p)=p\phi(C_{p^{m-1}})+(p-1)(p-2). $
\end{enumerate}

\end{lem}
\begin{proof}
\begin{enumerate}
\item[(i)] It is clear that $ \phi(C_{p^m})=\Sigma_{g\in C_{p^m}}\phi(o(g))=\Sigma_{g\in C_{p^{m-1}}}\phi(o(g))+(p^m-p^{m-1})^2$. The proof of the first part is complete.\\
\item[(ii)] Let $ A=\cup _{x\in C_{p^{m-1}}}(x,0) $, $ B=\cup_{0\neq a\in C_p}(0,a) $ and $ C=\cup_{0\neq a\in C_p}\cup_{0\neq x\in C_{p^{m-1}}}(x,a)$. Obviously they are a partition for $ C_{p^{m-1}}\times C_p $ and so $ C_{p^{m-1}}\times C_p= A\cup B\cup C $. It is clear that $ \phi(A)=\phi(C_{p^{m-1}}) $ and $ \phi(B)=\phi(C_p)-1 $. If $ (x,a)\in C $, then $ o\big( (x,a)\big)=o(x) $. Since for  each $ a \in C_p $ we have $ \phi(\cup_{0\neq x\in C_{p^{m-1}}}(x,a))= \phi(C_{p^{m-1}})-1$, this yields that $ \phi(C)=\Sigma_{0\neq a\in C_p}\Sigma_{0\neq x\in C_{p^{m-1}}}\phi(o(x))=(p-1)(\phi(C_{p^{m-1}})-1) $. Thus we have $ \phi(C_{p^{m-1}}\times C_p)=\phi(A)+\phi(B)+\phi(C)=\phi(C_{p^{m-1}})+\phi(C_p)-1+(p-1)(\phi(C_{p^{m-1}})-1)$. We also have $ \phi(C_p)=\Sigma_{x\in C_p}\phi(o(x))=(p-1)^2+1 $; hence, $\phi(C_{p^{m-1}}\times C_p)=p\phi(C_{p^{m-1}})+(p-1)(p-2)$.
\end{enumerate}
\end{proof}

\begin{lem}
Let $ p $ be a prime number and $ m\geq 2 $. Then
\begin{equation*}
(p-2)\phi(C_{p^{m-1}}\times C_p) < \phi(C_{p^m})< p\phi(C_{p^{m-1}}\times C_p).
\end{equation*}
\end{lem}
\begin{proof}
It follows from Lemma \ref{Mj} that
\begin{eqnarray*}
\phi(C_{p^m})=\phi(C_{p^{m-1}})+p^{2m-2}(p-1)^2.
\end{eqnarray*}
According to Lemma 2.5 in \cite{Kevin}, we have $ \phi(C_{p^{m-1}})=\dfrac{p^{2m-2}(p-1)+2}{p+1} $. Therefore
\begin{eqnarray*}
\phi(C_{p^m})&=&\phi(C_{p^{m-1}})+p^{2m-2}(p-1)^2 \leq \phi(C_{p^{m-1}})+\dfrac{(p^{2m-2}(p-1)+2)(p^2-1)}{p+1}
\\
&=&\phi(C_{p^{m-1}})+\phi(C_{p^{m-1}})(p^2-1)=p^2\phi(C_{p^{m-1}}).
\end{eqnarray*}
By Lemma \ref{Mj} we have
\begin{eqnarray*}
p^2\phi(C_{p^{m-1}})\leq p\phi(C_{p^{m-1}}\times C_p),
\end{eqnarray*}
and it completes the proof of the right side of the inequality.
\par By Lemma \ref{Mj} we have
\begin{eqnarray*}
\dfrac{\phi(C_{p^m})}{\phi(C_{p^{m-1}}\times C_p)}&=&\dfrac{\phi(C_{p^{m-1}})+\phi(p^m)^2}{p\phi(C_{p^{m-1}})+(p-1)(p-2)}> \dfrac{\phi(p^m)^2}{p(\phi(C_{p^{m-1}})+p)}
\\
&=&\dfrac{p^{2m-3}(p-1)^2}{\phi(C_{p^{m-1}})+p}> \dfrac{p^{2m-3}(p-1)^2}{p^{m-1}(p^{m-1}-p^{m-2})+p^{2m-3}}
\\
&=&\dfrac{(p-1)^2}{p}>p-2,
\end{eqnarray*}

and we get the result.
\end{proof}

\begin{cor}\label{Vj}
Let $ p $ and $ q $ be two prime numbers where $ p<q $  and $ t,m \geq2$. Then
\begin{equation*}
\frac{\phi(C_{p^m})}{\phi(C_{p^{m-1}}\times C_p)}<\dfrac{\phi(C_{q^t})}{\phi(C_{q^{t-1}}\times C_q)}.
\end{equation*}
\end{cor}
\begin{proof}
The right side of the inequality is greater than $ q-2 $ by previous lemma and since $ p<q$ are odd primes, we have $ p\leq q-2 $. According to previous lemma, the left side is less than $ p $. This completes the proof.
\end{proof}
%\bigskip
\textbf{\\Proof of the Main Theorem.}
Let $ G $ be a non-cyclic nilpotent group of order $ n $. Since $ G $ is a nilpotent group of order $ n $, we have
$ G\cong P_1\times P_2\times\cdots\times P_k $, where $ P_i $ is the unique Sylow $ p_i $-subgroup of $ G $ of order $ p^{\alpha_i}_i$. Therefore $ \phi(G)=\phi(P_1)\phi(P_2)\cdots\phi(P_k) $. Because we assume that $ G $ is not cyclic, at least one of the Sylow subgroups, say $ P_j $, is not cyclic. Hence we have $ \phi(P_j)\leq \phi(C_{ p^{\alpha_{j-1}}_j}\times C_{p_j}) $ by Proposition \ref{Bj}. It follows from Lemma 3.6 in \cite{Kevi} that  $ \phi(P_i)\leq \phi(C_{ p^{\alpha_i}_i}) $, for $ 1\leq i\leq k $. Therefore, we have
\begin{eqnarray*}
\phi(G)&\leq& \phi(C_{ p^{\alpha_1}_1})\cdots\phi(C_{ p^{\alpha_{j-1}}_{j-1}})\phi(C_{p^{\alpha_j}_j})\phi(C_{ p^{\alpha_{j+1}}_{j+1}})\cdots\phi(C_{ p^{\alpha_k}_k})
\\
&=&\phi(C_{\frac{n}{p^{\alpha_j}_j}})\phi(C_{ p^{\alpha_{j-1}}_j}\times C_{p_j})
\\
&=&\dfrac{\phi(C_n)}{\phi(C_{p^{\alpha_j}_j})}\phi(C_{p^{\alpha_{j-1}}_j}\times C_{ p_j})
\\
&\leq &\dfrac{\phi(C_n)\phi(C_{p^{\alpha_{s-1}}_s}\times C_{ p_s})}{\phi(C_{p^{\alpha_s}_s})}=\phi(C_{\frac{n}{p_s}}\times C_{p_s}),
\end{eqnarray*}
where the last inequality holds by Corollary \ref{Vj}. So the proof is complete.

\begin{cor}
Let $ n=p^{\alpha_1}_1p^{\alpha_2}_2\cdots p^{\alpha_k}_k $ be a positive odd integer which is not square free and $  p_1 < p_2 < \cdots< p_k  $ are primes. Set $  s = min\{1,\ldots,k\}$ such that $ \alpha_s>1 $. Then among directed power graph of non-cyclic nilpotent groups of order $ n $, the directed power graph $ \overrightarrow{\mathcal{G}}(C_{\frac{n}{p_s}} \times C_{p_s})$ has the maximum number of undirected edges.

\end{cor}
\begin{prop}
 Let $ G $ be a non-cyclic $p$-group of order $ p^n $ whose undirected power graph has the maximum number of edges among all non-cyclic p-groups of order $ p^n $. Then
\begin{enumerate}
\item[(i)]
if $ p $ is odd, then $G\cong C_{p^{m-1}}\times C_p $ or $G\cong M_{n,p}$;
\item[(ii)]
if $ p=2 $ and $ n\neq 3 $, then $ G\cong C_{2^{n-1}}\times C_2 $;
\item[(iii)]
if $ p^n=8 $, then $ G\cong Q_8 $.
\end{enumerate}
\end{prop}

\begin{proof}
 It follows from Theorem 4.2 in \cite{Ivy} that

\begin{eqnarray*}
 \vert E(G)\vert =\dfrac{1}{2}\Sigma_{g\in G}(2o(g)-\phi(o(g))-1).
\end{eqnarray*}
Since $ G $ is a finite $p$-group, for every non-identity element $ g\in G $ we have $ \phi(o(g))=o(g)-\dfrac{o(g)}{p} $. Therefore
\begin{eqnarray*}
 \vert E(G)\vert &=&\dfrac{1}{2}[\Sigma_{g\in G-\{ e\}}(2o(g)-(o(g)-\dfrac{o(g)}{p}))-\vert G\vert +1 ]\\
&=&\dfrac{1}{2}[\Sigma_{g\in G}o(g)+\dfrac{1}{p}\Sigma_{g\in G}o(g)-\vert G\vert -\dfrac{1}{p}].
\end{eqnarray*}
If $ p $ is an odd prime number, then $ E(G) $ has the maximum value when $ \Sigma_{g\in G} o(g)$ has its maximum value, too. By Proposition 3.2 in \cite{SKeller}, we know that if a non-cyclic $p$-group $ G $ of order $ p^n $ has the maximum value of $ \Sigma_{g\in G} o(g)$, then $ G\cong C_{p^{n-1}}\times C_p $ or $ G\cong M_{n,p }$. It completes the proof of (i).
\\
If $ p=2 $ and $ n\neq 3 $, then according to Proposition 3.2 in \cite{SKeller}, $ \Sigma_{g\in G}o(g) $ has the maximum value when $ G\cong  C_{2^{m-1}}\times C_2 $.
\\
If $ p^n=8 $, then by Proposition 3.2 in \cite{SKeller}, $ \Sigma_{g\in G}o(g) $ has the maximum value when $ G\cong  Q_8 $.

\end{proof}
\begin{conj}
Let $ n=p^{\alpha_1}_1p^{\alpha_2}_2\cdots p^{\alpha_k}_k $ be a positive odd integer which is not square free and $  p_1 < p_2 < \cdots < p_k  $ are primes. Set $  s = min\{1,\ldots,k\}$ such that $ \alpha_s>1 $. Among non-cyclic nilpotent groups of order $ n $, does the undirected power graph $ \mathcal{G}(C_{\frac{n}{p_s}}\times C_{p_s})  $ have the maximum number of edges?
\end{conj}

\end{document}